\documentclass[11pt]{article}

\usepackage[T1]{fontenc}
\usepackage[utf8]{inputenc}
\usepackage{lmodern}
\usepackage{amsmath,amssymb,amsthm,mathtools}
\usepackage[a4paper,margin=30mm]{geometry}
\usepackage{microtype}
\usepackage[hidelinks]{hyperref}

\newtheorem{theorem}{Theorem}[section]
\newtheorem{lemma}[theorem]{Lemma}
\newtheorem{corollary}[theorem]{Corollary}

\newcommand{\R}{\mathbb{R}}
\newcommand{\ftensor}{\mathbin{\bar{\otimes}}}
\newcommand{\ptensor}{\mathbin{\widehat{\otimes}_{|\pi|}}}
\newcommand{\one}{\mathbf{1}}

\title{Tensor Products of Almost \(f\)-Algebras: A Counterexample}
\author{Mohamed Amine Ben Amor\thanks{%
Research Laboratory of Algebra, Topology, Arithmetic, and Order,
Department of Mathematics, Faculty of Mathematical, Physical and Natural Sciences of Tunis,
Tunis El Manar University, Tunisia.}}
\date{}

\begin{document}

\maketitle

\begin{abstract}
We show that the natural tensor product multiplication of two Archimedean
almost \(f\)-algebras does not, in general, extend to their Fremlin tensor
product. The two algebras in the counterexample have the same underlying
vector lattice \(C[0,1]\). Their multiplications are explicit, positive and
nilpotent of index three. The obstruction is the fact that the function
\((s,t)\mapsto e^{st}\) does not belong to
\(C[0,1]\ftensor C[0,1]\). The same construction gives Banach almost
\(f\)-algebras and has consequences for previously published assertions on
Riesz and Fremlin projective tensor products of almost \(f\)-algebras.
\end{abstract}



\section{Introduction}

For Archimedean \(f\)-algebras, the Fremlin tensor product admits a natural
\(f\)-algebra multiplication extending
\[
 (a\otimes b)(c\otimes d)=ac\otimes bd;
\]
see \cite{AzouziBenAmorJaber,BuskesWickstead}. It is natural to ask whether
the same statement remains true for almost \(f\)-algebras.

Theorem~3.12 of \cite{Gok} gives an affirmative answer. Its proof uses the
claim that every symmetric bilinear map is orthosymmetric. This is false,
even for positive symmetric bilinear maps. For example, on \(\R^2\) with the coordinate order, let
\[
 e_1=(1,0),\qquad e_2=(0,1),
\]
and define
\[
 \Phi(x,y)=(x_1+x_2)(y_1+y_2),
 \qquad x=(x_1,x_2),\ y=(y_1,y_2).
\]
Then \(\Phi:\R^2\times\R^2\to\R\) is positive and symmetric. However,
\[
 e_1\perp e_2,
 \qquad
 \Phi(e_1,e_2)=1.
\]
Thus symmetry, even together with positivity, does not imply
orthosymmetry.

We prove that the statement itself is false. The example uses only the
vector lattice \(C[0,1]\). The main ingredient is a bounded positive operator
\(S_0\) for which a simple element of the Fremlin tensor product is sent,
after acting in one variable, to the function \(e^{st}\), which does not
belong to that tensor product. We then build two almost \(f\)-algebra
multiplications on \(C[0,1]\) which force this operator to appear as
multiplication by the positive tensor \(\one\otimes\one\).

\section{The obstruction on \(C[0,1]\ftensor C[0,1]\)}

Put
\[
 I=[0,1],\qquad X=C(I),\qquad H=X\ftensor X.
\]
We identify \(H\) with the Riesz subspace of \(C(I^2)\) generated by the
algebraic tensor product \(X\otimes X\); see \cite{Fremlin1972} and
\cite[Proposition~3.1]{BuskesWickstead}.

We use the following standard form of Fremlin's positive universal property:
if \(G\) is uniformly complete, then every positive bilinear map
\(X\times X\to G\) has a unique positive linear extension from \(H\) to
\(G\); see \cite[Theorem~5.3]{Fremlin1972}.

We shall also use the following elementary fact about operators acting in one
variable.

\begin{lemma}
Let \(L,M:X\to X\) be norm-bounded linear operators. For \(q\in C(I^2)\), define
\[
 (L_xq)(s,t)=L(q(\cdot,t))(s),
 \qquad
 (M_yq)(s,t)=M(q(s,\cdot))(t).
 \]
Then \(L_x\) and \(M_y\) are bounded linear operators on \(C(I^2)\), with
\(\|L_x\|\leq\|L\|\) and \(\|M_y\|\leq\|M\|\). If \(L\) and
\(M\) are positive, then \(L_x\) and \(M_y\) are positive and hence order
bounded. Moreover,
\[
 L_xM_y=M_yL_x.
\]
\end{lemma}

\begin{proof}
For \(q\in C(I^2)\), the map \(t\mapsto q(\cdot,t)\) is continuous from
\(I\) into \(X\) for the supremum norm. Hence
\(t\mapsto L(q(\cdot,t))\) is continuous from \(I\) into \(X\). It follows
that \((s,t)\mapsto L(q(\cdot,t))(s)\) is continuous. Also
\[
 \|L_xq\|_\infty\leq\|L\|\,\|q\|_\infty.
\]
The same argument applies to \(M_y\). If \(L\ge0\), then
\(q\ge0\) implies \(L_xq\ge0\), so \(L_x\) is positive. In particular, if
\(|q|\le u\) with \(u\in C(I^2)_+\), then
\[
 |L_xq|\le L_x|q|\le L_xu,
\]
which proves order boundedness. The same argument applies to \(M_y\). Thus
order boundedness and norm boundedness are both available here, but they are
separate properties.

For an elementary tensor \(q=f\otimes g\),
\[
 L_xM_y(f\otimes g)=Lf\otimes Mg=M_yL_x(f\otimes g).
\]
The algebraic tensor product \(X\otimes X\) is uniformly dense in
\(C(I^2)\), and the two operators are bounded. Hence the equality holds on
all of \(C(I^2)\).
\end{proof}

\begin{lemma}
The function
\[
 k:I^2\longrightarrow\R,\qquad k(s,t)=e^{st},
\]
does not belong to \(H=X\ftensor X\).
\end{lemma}

\begin{proof}
Let
\[
 V=X\otimes X\subset C(I^2).
\]
By the standard description of the vector sublattice generated by a vector
subspace, the vector sublattice \(\operatorname{lat}(V)\) generated by \(V\)
consists exactly of the elements of the form
\[
 \bigvee_{i=1}^{m}u_i-\bigvee_{j=1}^{n}v_j,
 \qquad u_1,\ldots,u_m,v_1,\ldots,v_n\in V;
 \tag{2.1}
\]
see \cite[p.~204, Exercise~8(b)]{AliprantisBurkinshaw}. Since \(H\) is the
vector sublattice of \(C(I^2)\) generated by \(V\), we have
\(H=\operatorname{lat}(V)\).

Assume that \(k\in H\). By \emph{(2.1)}, there exist
\(u_1,\ldots,u_m,v_1,\ldots,v_n\in V\) such that
\[
 k=\bigvee_{i=1}^{m}u_i-\bigvee_{j=1}^{n}v_j.
 \tag{2.2}
\]
For \(1\le i\le m\) and \(1\le j\le n\), put
\[
 w_{ij}=u_i-v_j\in V.
\]
Fix \((s,t)\in I^2\). Since the suprema in \emph{(2.2)} are finite, there are
indices \(i=i(s,t)\) and \(j=j(s,t)\) such that
\[
 \bigvee_{r=1}^{m}u_r(s,t)=u_i(s,t),
 \qquad
 \bigvee_{q=1}^{n}v_q(s,t)=v_j(s,t).
\]
Hence
\[
 k(s,t)=u_i(s,t)-v_j(s,t)=w_{ij}(s,t).
 \tag{2.3}
\]
Thus, at every point of \(I^2\), the function \(k\) agrees with one of the
finitely many functions \(w_{ij}\in V\).

For each pair \((i,j)\), define
\[
 C_{ij}=\{(s,t)\in I^2:k(s,t)=w_{ij}(s,t)\}.
\]
Each \(C_{ij}\) is closed, and by \emph{(2.3)} the finitely many sets
\(C_{ij}\) cover \(I^2\). Since \(I^2\) is a complete metric space, the Baire
category theorem implies that at least one of them has non-empty interior.
Consequently, for some pair \((i_0,j_0)\), there exist non-empty open intervals
\(U,V_0\subset I\) such that
\[
 U\times V_0\subset C_{i_0j_0}.
\]
Since \(w_{i_0j_0}\in X\otimes X\), there are \(N\in\mathbb{N}\) and functions
\(f_1,\ldots,f_N,g_1,\ldots,g_N\in X\) such that
\[
 w_{i_0j_0}(s,t)=\sum_{q=1}^{N}f_q(s)g_q(t).
\]
Therefore, on \(U\times V_0\),
\[
 e^{st}=\sum_{q=1}^{N}f_q(s)g_q(t).
 \tag{2.4}
\]

Choose pairwise distinct points \(t_1,\ldots,t_{N+1}\in V_0\). For every
\(1\le j\le N+1\), equation \emph{(2.4)} gives
\[
 e^{t_js}=\sum_{q=1}^{N}g_q(t_j)f_q(s),
 \qquad s\in U.
\]
Hence the \(N+1\) functions \(s\mapsto e^{t_js}\) belong to the vector space
spanned by \(f_1|_U,\ldots,f_N|_U\), which has dimension at most \(N\).
On the other hand, their Wronskian is
\[
 \det\!\left(\frac{d^{r-1}}{ds^{r-1}}e^{t_js}\right)_{r,j=1}^{N+1}
 =e^{(t_1+\cdots+t_{N+1})s}
   \prod_{1\le i<j\le N+1}(t_j-t_i),
\]
which is non-zero for every \(s\in U\). Thus these functions are linearly
independent on \(U\), a contradiction.
\end{proof}

For \(s\in I\), define
\[
 a(s)=\frac{e^s+1+s}{2},
 \qquad
 b(s)=\frac{1+(1-s)e^s}{2}.
\]
Both functions are positive. Define \(S_0:X\to X\) by
\[
 (S_0f)(s)
 =a(s)f(0)+b(s)f(1)
 +\frac{s^2}{2}\int_0^1 e^{sx}f(x)\,dx.
 \tag{2.5}
\]

We distinguish two notions which will be used below. A linear operator between
vector lattices is \emph{order bounded} if it maps every order-bounded set
into an order-bounded set. Norm boundedness refers here to the supremum norm
on \(C(I)\). These are different notions; for the operators used in this
paper we verify both properties separately.

\begin{lemma}
The formula \emph{(2.5)} defines a positive linear operator
\(S_0:X\to X\). The operator \(S_0\) is order bounded. It is also norm
bounded and
\[
 \|S_0\|=e+1.
 \tag{2.6}
\]
If
\[
 q(x,t)=|x-t|,
\]
then
\[
 (S_{0,x}q)(s,t)=e^{st},
 \qquad s,t\in I,
 \tag{2.7}
\]
where \(S_{0,x}\) is the norm-bounded operator on \(C(I^2)\) defined in
Lemma~2.1.
\end{lemma}

\begin{proof}
Let \(f\in X\). The function
\[
 (s,x)\longmapsto e^{sx}f(x)
\]
is continuous on the compact set \(I^2\). Therefore
\[
 s\longmapsto \int_0^1 e^{sx}f(x)\,dx
\]
is continuous. Since \(a\) and \(b\) are continuous, \emph{(2.5)} defines
an element \(S_0f\in X\). Linearity is immediate. If \(f\ge0\), every term
in \emph{(2.5)} is non-negative, so \(S_0f\ge0\). Thus \(S_0\) is positive.

We first verify order boundedness. Let \(u\in X_+\) and let \(f\in X\) satisfy
\(|f|\le u\). Positivity gives
\[
 |S_0f|\le S_0|f|\le S_0u.
\]
Hence
\[
 S_0([-u,u])\subset[-S_0u,S_0u],
\]
which proves that \(S_0\) is order bounded.

We now verify norm boundedness independently. Taking \(f=\one\) in
\emph{(2.5)}, and using
\[
 \frac{s^2}{2}\int_0^1e^{sx}\,dx
 =\frac{s}{2}(e^s-1)
 \quad(s>0),
\]
with the value at \(s=0\) obtained by continuity, we obtain
\[
 (S_0\one)(s)
 =a(s)+b(s)+\frac{s^2}{2}\int_0^1e^{sx}\,dx
 =e^s+1.
 \tag{2.8}
\]
For arbitrary \(f\in X\),
\[
 |f|\le\|f\|_\infty\one
\]
and positivity yields
\[
 |S_0f|
 \le \|f\|_\infty S_0\one.
\]
Consequently
\[
 \|S_0f\|_\infty
 \le(e+1)\|f\|_\infty.
\]
Thus \(\|S_0\|\le e+1\). Equality follows from \emph{(2.8)}, since
\(\|\one\|_\infty=1\) and \(\|S_0\one\|_\infty=e+1\). This proves
\emph{(2.6)}.

It remains to prove \emph{(2.7)}. For fixed \(s\in I\), put
\[
 F_s(t)
 =a(s)t+b(s)(1-t)
 +\frac{s^2}{2}\int_0^1e^{sx}|x-t|\,dx.
\]
For \(0<t<1\), split the integral at \(x=t\):
\[
 \int_0^1e^{sx}|x-t|\,dx
 =\int_0^t e^{sx}(t-x)\,dx
  +\int_t^1e^{sx}(x-t)\,dx.
\]
Differentiating with respect to \(t\) gives
\[
 \frac{d}{dt}\int_0^1e^{sx}|x-t|\,dx
 =\int_0^t e^{sx}\,dx-\int_t^1e^{sx}\,dx,
\]
and a second differentiation gives
\[
 \frac{d^2}{dt^2}\int_0^1e^{sx}|x-t|\,dx=2e^{st}.
\]
Hence
\[
 F_s''(t)=s^2e^{st},\qquad 0<t<1.
 \tag{2.9}
\]
At \(t=0\),
\[
 F_s(0)
 =b(s)+\frac{s^2}{2}\int_0^1xe^{sx}\,dx=1,
 \tag{2.10}
\]
and at \(t=1\),
\[
 F_s(1)
 =a(s)+\frac{s^2}{2}\int_0^1(1-x)e^{sx}\,dx=e^s.
 \tag{2.11}
\]
The identities \emph{(2.10)}--\emph{(2.11)} follow by elementary integration
by parts; they also remain valid at \(s=0\) by direct substitution. The
function \(t\mapsto e^{st}\) has the same second derivative
\(s^2e^{st}\) and the same values at \(0\) and \(1\). Thus
\(F_s(t)-e^{st}\) is affine and vanishes at both endpoints. It is therefore
identically zero. This proves \emph{(2.7)}.
\end{proof}

\section{Two almost \(f\)-algebras on \(C[0,1]\)}

We now separate the part of a function which is read by the multiplication
from the part where the product is written. Define
\[
 P:X\longrightarrow X,
 \qquad
 (Pf)(u)=f(u/3),
 \qquad f\in X,\ u\in I.
 \tag{3.1}
\]
Thus \(P\) reads only the restriction of \(f\) to \([0,1/3]\). Since the
map \(u\mapsto u/3\) is continuous, \(Pf\in X\). Moreover,
\[
 |Pf|=P|f|,
\]
so \(P\) is a Riesz homomorphism, in particular a positive linear operator.
If \(|f|\le u\) with \(u\in X_+\), then
\[
 |Pf|\le Pu,
\]
so \(P\) is order bounded. On the other hand,
\[
 \|Pf\|_\infty\le\|f\|_\infty,
\]
and \(P\one=\one\); hence \(P\) is norm bounded and
\[
 \|P\|=1.
 \tag{3.2}
\]

To write the output away from \([0,1/3]\), put
\[
 \sigma(r)=\max\{3r-2,0\},
\]
and define the continuous function
\[
 \eta(r)=
 \begin{cases}
 0, & 0\le r\le \frac12,\\[2mm]
 6r-3, & \frac12\le r\le \frac23,\\[2mm]
 1, & \frac23\le r\le1.
 \end{cases}
\]
Thus \(0\le\eta\le1\), \(\eta=0\) on \([0,1/2]\), and \(\eta=1\) on
\([2/3,1]\). Define
\[
 E:X\longrightarrow X,
 \qquad
 (Eg)(r)=\eta(r)g(\sigma(r)).
 \tag{3.3}
\]
Since \(\eta\) and \(\sigma\) are continuous, \(Eg\in X\). Furthermore,
\[
 |Eg|(r)=\eta(r)|g(\sigma(r))|=(E|g|)(r),
\]
so \(E\) is a Riesz homomorphism and hence positive. If \(|g|\le u\), then
\(|Eg|\le Eu\), proving order boundedness. Also
\[
 \|Eg\|_\infty\le\|g\|_\infty.
\]
Since \(E\one=\eta\) and \(\|\eta\|_\infty=1\),
\[
 \|E\|=1.
 \tag{3.4}
\]

The separation of the two intervals gives the key identity
\[
 PE=0.
 \tag{3.5}
\]
Indeed, for \(g\in X\) and \(u\in I\),
\[
 (PEg)(u)
 =(Eg)(u/3)
 =\eta(u/3)g(\sigma(u/3)).
\]
But \(0\le u/3\le1/3<1/2\), so \(\eta(u/3)=0\). Hence
\((PEg)(u)=0\) for every \(u\in I\), which proves \emph{(3.5)}.

Finally, let
\[
 J=[2/3,1].
\]
For \(r\in J\), we have \(\eta(r)=1\) and \(\sigma(r)=3r-2\). Therefore
\[
 (Eg)(r)=g(3r-2),
 \qquad r\in J.
 \tag{3.6}
\]

On the vector lattice \(X\), define two products by
\[
 f\cdot_A g
 =E\bigl(S_0((Pf)(Pg))\bigr),
 \tag{3.7}
\]
and
\[
 f\cdot_B g
 =E\bigl((Pf)(Pg)\bigr).
 \tag{3.8}
\]

\begin{lemma}
With the products \emph{(3.7)} and \emph{(3.8)}, respectively,
\(A=(X,\cdot_A)\) and \(B=(X,\cdot_B)\) are Archimedean almost
\(f\)-algebras. In both algebras every product of three elements is zero.
\end{lemma}

\begin{proof}
Recall that an almost \(f\)-algebra is, in particular, a lattice-ordered
associative algebra; see \cite{BernauHuijsmans}. Thus associativity of the
products in \emph{(3.7)} and \emph{(3.8)} has to be verified explicitly.
The underlying ordered vector space is \(X=C(I)\), hence it is an
Archimedean vector lattice. We verify the algebraic and order properties of
the two multiplications separately.

First, the map
\[
 (f,g)\longmapsto (Pf)(Pg)
\]
is bilinear and symmetric, because \(P\) is linear and the multiplication in
\(C(I)\) is bilinear and commutative. Since \(E\) and \(S_0\) are linear,
\emph{(3.7)} and \emph{(3.8)} therefore define bilinear and commutative
products on \(X\).

If \(f,g\ge0\), then \(Pf,Pg\ge0\), and consequently
\((Pf)(Pg)\ge0\). Since both \(S_0\) and \(E\) are positive,
\[
 f\cdot_Ag=E(S_0((Pf)(Pg)))\ge0,
 \qquad
 f\cdot_Bg=E((Pf)(Pg))\ge0.
\]
Thus both products are positive.

We next prove associativity. The key point is that every product is annihilated
by \(P\). Indeed, every product lies in the range of \(E\), and
\emph{(3.5)} gives
\[
 \begin{aligned}
 P(f\cdot_Ag)
 &=PE\bigl(S_0((Pf)(Pg))\bigr)=0,\\
 P(f\cdot_Bg)
 &=PE\bigl((Pf)(Pg)\bigr)=0.
 \end{aligned}
 \tag{3.9}
\]
Hence, for \(f,g,h\in X\),
\[
 \begin{aligned}
 (f\cdot_Ag)\cdot_Ah
 &=E\!\left(S_0\bigl(P(f\cdot_Ag)\,Ph\bigr)\right)=0,\\
 f\cdot_A(g\cdot_Ah)
 &=E\!\left(S_0\bigl(Pf\,P(g\cdot_Ah)\bigr)\right)=0.
 \end{aligned}
\]
Therefore both possible parenthesizations of a product of three elements of
\(A\) are zero:
\[
 (f\cdot_A g)\cdot_A h=0=f\cdot_A(g\cdot_A h).
\]
Hence \(\cdot_A\) is associative. The same argument, with \(S_0\) omitted,
gives
\[
 (f\cdot_B g)\cdot_B h=0=f\cdot_B(g\cdot_B h),
\]
so \(\cdot_B\) is associative as well. Thus every product of three
elements is zero in both algebras.

Finally, let \(f\perp g\) in \(X\). Since disjointness in \(C(I)\) is
pointwise,
\[
 f(x)g(x)=0,
 \qquad x\in I.
\]
Therefore, for every \(u\in I\),
\[
 (Pf)(u)(Pg)(u)
 =f(u/3)g(u/3)=0.
\]
Thus \((Pf)(Pg)=0\), and \emph{(3.7)}--\emph{(3.8)} give
\[
 f\cdot_Ag=0,
 \qquad
 f\cdot_Bg=0.
\]
This is the defining disjointness condition for an almost \(f\)-algebra;
see \cite{BernauHuijsmans}. Hence \(A\) and \(B\) are Archimedean almost
\(f\)-algebras.
\end{proof}

\section{The counterexample}

Since the underlying vector lattice of both \(A\) and \(B\) is \(X\), their
Fremlin tensor product is the same vector lattice
\[
 H=X\ftensor X.
\]

\begin{theorem}
There is no positive bilinear multiplication \(\star\) on \(H\) such that
\[
 (f\otimes g)\star(h\otimes k)
 =(f\cdot_A h)\otimes(g\cdot_B k)
 \tag{4.1}
\]
for all \(f,h\in A\) and \(g,k\in B\).
\end{theorem}

\begin{proof}
Assume that such a multiplication exists. Put
\[
 p=\one\otimes\one\in H_+,
\]
and define
\[
 M:H\to H,
 \qquad
 Mz=z\star p.
\]
Since \(\star\) is positive, \(M\) is positive. By \emph{(3.7)} and
\emph{(3.8)}, for \(f,g\in X\),
\[
 M(f\otimes g)=L_Af\otimes L_Bg,
 \tag{4.2}
\]
where
\[
 L_A=ES_0P,
 \qquad
 L_B=EP.
 \tag{4.3}
\]

By the first lemma of Section~2, the operators \((L_A)_x\) and \((L_B)_y\) are well-defined, bounded and positive, and operators acting in different variables commute. Hence
\[
 U=(L_A)_x(L_B)_y:C(I^2)\to C(I^2)
\]
is a bounded positive operator. On elementary tensors,
\[
 U(f\otimes g)=L_Af\otimes L_Bg.
\]
Let \(J:H\to C(I^2)\) be the canonical inclusion. Then \(J\circ M\) and
\(U|_H\) are positive linear extensions of the same positive bilinear map
\[
 (f,g)\longmapsto L_Af\otimes L_Bg
\]
from \(X\times X\) into \(C(I^2)\). By Fremlin's positive universal
property,
\[
 J\circ M=U|_H.
 \tag{4.4}
\]

Let
\[
 \theta(x)=\min\{3x,1\},\qquad x\in I,
\]
and put
\[
 h=|\theta\otimes\one-\one\otimes\theta|\in H.
\]
Since \(P\theta(u)=u\) and \(P\one=\one\),
\[
 (P_xP_yh)(u,v)=|u-v|.
\]
By the commutation result proved in Section~2 and by \emph{(4.3)},
\[
 Uh=E_xE_yS_{0,x}(P_xP_yh).
\]
Using \emph{(2.7)},
\[
 Uh=E_xE_y(e^{uv}).
 \tag{4.5}
\]
By \emph{(3.6)}, for \(s,t\in J\),
\[
 (Uh)(s,t)=e^{(3s-2)(3t-2)}.
 \tag{4.6}
\]

By \emph{(4.4)}, \(Uh=J(Mh)\), hence \(Uh\in H\). Define
\[
 (Rq)(u,v)
 =q\left(\frac{u+2}{3},\frac{v+2}{3}\right),
 \qquad q\in C(I^2).
 \tag{4.7}
\]
The map \(R\) is a Riesz homomorphism on \(C(I^2)\). Moreover,
\[
 R(f\otimes g)
 =\left(f\circ\alpha\right)\otimes\left(g\circ\alpha\right),
 \qquad
 \alpha(u)=\frac{u+2}{3},
\]
so \(R\) maps the generating set \(X\otimes X\) into itself. Since \(R\) is a Riesz homomorphism and \(H\) is the Riesz subspace generated by \(X\otimes X\), we have \(R(H)\subset H\). Applying \(R\) to \emph{(4.6)} gives
\[
 R(Uh)(u,v)=e^{uv}.
\]
Thus \(e^{uv}\in H\), contradicting Lemma~2.1.
\end{proof}

Hence the Fremlin tensor product of two Archimedean almost \(f\)-algebras
need not admit the natural tensor product multiplication. In particular,
Theorem~3.12 of \cite{Gok} is false.

The example can also be made Banach. The operators \(P\) and \(E\) have
norm at most one. Put
\[
 c=\max\{1,\|S_0\|\}
\]
and equip \(A\) with the lattice norm
\[
 \|f\|_A=c\|f\|_\infty.
\]
Keep the usual supremum norm on \(B\). Then
\[
 \|f\cdot_A g\|_A
 \le \|f\|_A\|g\|_A,
 \qquad
 \|f\cdot_B g\|_\infty
 \le \|f\|_\infty\|g\|_\infty.
\]
Thus \(A\) and \(B\) are Banach almost \(f\)-algebras.

\begin{corollary}
There are Banach almost \(f\)-algebras \(A\) and \(B\) such that their Riesz
tensor product cannot be a normed almost \(f\)-algebra under the natural
tensor product multiplication. Moreover, the canonical copy of
\(A\ftensor B\) in the Fremlin projective tensor product \(A\ptensor B\) is
not a subalgebra for the natural projective multiplication.
\end{corollary}

\begin{proof}
The first assertion follows from the theorem. By \cite{Fremlin1974,Jaber},
\(A\ptensor B\) is a Banach lattice algebra whose multiplication extends the
multiplication on elementary tensors, and \(A\ftensor B\) is canonically a
norm-dense Riesz subspace. If this copy of \(A\ftensor B\) were a subalgebra,
the restricted multiplication would contradict the theorem.
\end{proof}

Consequently, Theorem~3.13 of \cite{Gok} is false. The assertion in
Theorem~3.14 of \cite{Gok} that the Fremlin projective tensor product contains
the Riesz tensor product as an almost \(f\)-subalgebra is also false. The
present example does not decide the separate assertion that the projective
completion itself is always a Banach almost \(f\)-algebra.

\section*{Acknowledgments}

This note grew out of a problem that the author worked on for several years
and later abandoned after many unsuccessful attempts. The author thanks
Page Thorn for a discussion at COSA 2026 which brought the problem back and
led to the counterexample given here. The author also thanks CIRM for hosting
COSA 2026 and for providing an environment which made these exchanges
possible.

\section*{Declarations}

\noindent\textbf{Competing Interests.}
The author declares that he has no competing interests.

\medskip
\noindent\textbf{Funding Information.}
Not applicable.

\medskip
\noindent\textbf{Author Contribution.}
M.A.B.A. conceived the study, developed the counterexample and proofs, and
wrote the manuscript.

\medskip
\noindent\textbf{Data Availability Statement.}
Not applicable. No datasets were generated or analysed during the current
study.


\begin{thebibliography}{99}

\bibitem{AzouziBenAmorJaber}
Y.~Azouzi, M.~A.~Ben Amor and J.~Jaber,
\newblock The tensor product of \(f\)-algebras,
\newblock \emph{Quaest. Math.} \textbf{41} (2018), no.~3, 359--369.
\newblock \href{https://doi.org/10.2989/16073606.2017.1382018}
{doi:10.2989/16073606.2017.1382018}.

\bibitem{AliprantisBurkinshaw}
C.~D.~Aliprantis and O.~Burkinshaw,
\newblock \emph{Positive Operators},
\newblock Springer, Dordrecht, 2006, reprint of the 1985 original.


\bibitem{BernauHuijsmans}
S.~J.~Bernau and C.~B.~Huijsmans,
\newblock Almost \(f\)-algebras and \(d\)-algebras,
\newblock \emph{Math. Proc. Cambridge Philos. Soc.}
\textbf{107} (1990), 287--308.
\newblock \href{https://doi.org/10.1017/S0305004100068560}
{doi:10.1017/S0305004100068560}.

\bibitem{BuskesWickstead}
G.~J.~H.~M.~Buskes and A.~W.~Wickstead,
\newblock Tensor products of \(f\)-algebras,
\newblock \emph{Mediterr. J. Math.} \textbf{14} (2017), Article 63.
\newblock \href{https://doi.org/10.1007/s00009-017-0841-x}
{doi:10.1007/s00009-017-0841-x}.

\bibitem{Fremlin1972}
D.~H.~Fremlin,
\newblock Tensor products of Archimedean vector lattices,
\newblock \emph{Amer. J. Math.} \textbf{94} (1972), 777--798.
\newblock \href{https://doi.org/10.2307/2373758}
{doi:10.2307/2373758}.

\bibitem{Fremlin1974}
D.~H.~Fremlin,
\newblock Tensor products of Banach lattices,
\newblock \emph{Math. Ann.} \textbf{211} (1974), 87--106.
\newblock \href{https://doi.org/10.1007/BF01344164}
{doi:10.1007/BF01344164}.

\bibitem{Gok}
O.~Gok,
\newblock On the Fremlin projective tensor product of Banach \(d\)-algebras
and almost \(f\)-algebras,
\newblock \emph{International Mathematical Forum}
\textbf{18} (2023), no.~3, 111--120.
\newblock \href{https://doi.org/10.12988/imf.2023.912391}
{doi:10.12988/imf.2023.912391}.

\bibitem{Jaber}
J.~Jaber,
\newblock The Fremlin projective tensor product of Banach lattice algebras,
\newblock \emph{J. Math. Anal. Appl.} \textbf{488} (2020), 123993.
\newblock \href{https://doi.org/10.1016/j.jmaa.2020.123993}
{doi:10.1016/j.jmaa.2020.123993}.


\end{thebibliography}
\end{document}